\documentclass[a4paper,12pt]{article}
\usepackage{amsmath}
\usepackage{amssymb}
\usepackage{amsfonts}
\usepackage{amsthm}
\usepackage{bm}
\usepackage{braket}
\usepackage{mathtools}
\usepackage[truedimen,top=45truemm,bottom=45truemm,left=30truemm,right=30truemm]{geometry}

\allowdisplaybreaks[1]

\newtheorem{thm}{Theorem}
\newtheorem{prop}[thm]{Proposition}
\newtheorem{lem}[thm]{Lemma}
\newtheorem{cor}[thm]{Corollary}
\newtheorem{rem}[thm]{Remark}
\newtheorem{ex}[thm]{Example}
\newtheorem{mainthm}[thm]{Main Theorem}

\renewcommand{\labelenumi}{\rm(\roman{enumi})}

\numberwithin{thm}{section}
\numberwithin{equation}{section}

\newcommand{\N}{\mathbb{Z}_{\geq 0}}
\newcommand{\Z}{\mathbb{Z}}
\newcommand{\Q}{\mathbb{Q}}
\newcommand{\C}{\mathbb{C}}
\newcommand{\Qbar}{\overline{\mathbb{Q}}}
\newcommand{\Cbar}{\overline{C}}
\newcommand{\0}{\bm{0}}
\newcommand{\x}{\bm{x}}
\newcommand{\y}{\bm{y}}
\newcommand{\z}{\bm{z}}
\newcommand{\ba}{\bm{\alpha}}
\newcommand{\bbeta}{\bm{\beta}}
\newcommand{\bmu}{\bm{\mu}}
\newcommand{\bb}{\bm{b}}
\newcommand{\f}{\bm{f}}

\newcommand{\bj}{\bm{j}}
\newcommand{\m}{\bm{m}}
\newcommand{\bn}{\bm{n}}

\newcommand{\bI}{\bm{I}}
\newcommand{\bN}{\bm{N}}

\newcommand{\wa}{\widetilde{a}}
\newcommand{\wg}{\widetilde{g}}
\newcommand{\wh}{\widetilde{h}}
\newcommand{\wt}{\widetilde{\theta}}
\newcommand{\wG}{\widetilde{G}}
\newcommand{\wH}{\widetilde{H}}
\newcommand{\wR}{\widetilde{R}}
\newcommand{\wT}{\widetilde{\Theta}}
\newcommand{\cB}{\mathcal{B}}
\newcommand{\cL}{\mathcal{L}}
\newcommand{\cM}{\mathcal{M}}
\newcommand{\relmiddle}[1]{\mathrel{}\middle#1\mathrel{}}

\DeclareMathOperator*{\ord}{ord}

\def\hsymb#1{\mbox{\strut\rlap{\smash{\Huge$#1$}}\quad}}
\def\diag{\mathop{\rm diag}\nolimits}

\begin{document}

\setlength{\baselineskip}{19.4pt}

\title{Algebraic independence of the partial derivatives of certain functions with arbitrary number of variables}
\author{Haruki Ide and Taka-aki Tanaka}
\date{}
\maketitle

\begin{abstract}
We construct a complex entire function with arbitrary number of variables which has the following property: The infinite set consisting of all the values of all its partial derivatives of any orders at all algebraic points, including zero components, is algebraically independent. 
In Section~\ref{sec:2} of this paper, we develop a new technique involving linear isomorphisms and infinite products to replace the algebraic independence of the values of functions in question with that of 
functions easier to deal with. 
In Section~\ref{sec:2} and \ref{sec:3}, using the technique together with Mahler's method, we can reduce the algebraic independence of the infinite set mentioned above 
to the linear independence of certain rational functions 
modulo the rational function field of many variables. 
The latter one is solved by the discussions involving a certain valuation 
and a generic point in Section~\ref{sec:3} and \ref{sec:4}. 

\vspace{5pt}

\end{abstract}

\section{Introduction and the main results}\label{sec:1}
There is no general theory on the algebraic independence of the values of functions with several variables even if whose specializations are known to take algebraically independent values. 
For instance, there are severe difficulties in determining whether the set of the values $\{f(\x,\y)\mid \x\in S,\ \y\in T\}$ with $S\subset\C^m$ and  $T\subset\C^n$ of a given $m+n$ variable function $f(\x,\y)$ 
is algebraically independent or not when 
the sets of the specializations $\{f(\ba,\y)\mid \y\in T\}$ and $\{f(\x,\bbeta)\mid \x\in S\}$ for any fixed $\ba\in S$ and $\bbeta\in T$ are both algebraically independent. 
On the other hand, there are some complex entire functions of one variable introduced below each of which is known to have the following property: All its values as well as all the values of its derivatives of any orders at all nonzero algebraic numbers are algebraically independent. 
For the reason mentioned above, it is however difficult 
to extend this property to higher dimensions and obtain
a complex entire function of several variables having the following property: The infinite set consisting of all the values of all its partial derivatives of any orders at all algebraic points, including zero components, is algebraically independent. 
The main aim of this paper is to construct such a complex entire function with arbitrary number of variables. 

In previous results, such entire functions were known only for the case of one variable,
as far as the authors know, 
except the first author's result~\cite{Ide}, which treated that of two variables. 
First we mention the previous results on the case of one 
variable. 

Here and in what follows, let $a$ be an algebraic number with $0<|a|<1$. 
Define $\varphi_1(x)=\sum_{k=0}^\infty a^{k!}x^k$ and 
$\varphi_2(x)=\sum_{k=0}^\infty a^{d^k}x^k$, 
where $d\geq2$ is an integer. 
In what follows, we denote by $\N$ the set of nonnegative integers, by $\Qbar$ the field of algebraic numbers, by $F^\times$ the set of nonzero elements of a field $F$, and by $f^{(l)}(x)$ the derivative of a function $f(x)$ of order $l$ with respect to the complex variable $x$. 
Nishioka~\cite{Nishioka1986}, \cite{Nishioka1996} showed, 
respectively, that the infinite set 
$\{\varphi_1^{(l)}(\alpha)\mid l\in\N,\ \alpha\in\Qbar^\times
\}$ is algebraically independent and so is the infinite set $\{
\varphi_2^{(l)}(\alpha)\mid l\in\N,\ \alpha\in\Qbar^\times\}$. 
Obviously, well-known transcendental entire functions, e.g. the exponential function $e^x$, the sine function $\sin x$, and the reciprocal of the Gamma function $1/\Gamma(x)$, do not have 
the property such as $\varphi_1(x)$ and $\varphi_2(x)$. 
Another example of a complex entire function of one variable having this kind of property is given by infinite products generated by linear recurrences as follows. 

Let $\{R_k\}_{k\geq0}$ be a linear recurrence of nonnegative integers satisfying
\begin{equation}\label{eq:LRS}
R_{k+n}=c_1R_{k+n-1}+\cdots+c_nR_k\quad (k\geq 0),
\end{equation}
where $R_0,\ldots,R_{n-1}$ are not all zero and $c_1,\ldots, c_n$ are nonnegative integers with $c_n\neq0$. 
Define $$\varphi_3(y)\coloneqq\prod_{k=0}^\infty \left(1-a^{R_k}y\right).$$ 
If $\{R_k\}_{k\geq0}$ is a geometric progression, 
for example if $R_k=d^k\ (k\geq0)$ with $d$ as above, 
then 
$\varphi_3(y)$ does not have the property such as $\varphi_1(x)$ and $\varphi_2(x)$ (see Kurosawa, Tachiya, and Tanaka~\cite[Introduction]{KTT}). 
Hence we suppose throughout this paper that $\{R_k\}_{k\geq0}$ is not a geometric progression and so $n\geq2$.

Moreover, to ensure the convergence of $\varphi_3(y)$, we assume one of the simplest sufficient conditions under which 
the same value appears in $\{R_k\}_{k\geq0}$ only finitely many times. 
Define the polynomial associated with \eqref{eq:LRS} by
\begin{equation}\label{eq:Phi(X)}
\Phi(X)\coloneqq X^n-c_1X^{n-1}-\cdots-c_n.
\end{equation}
It is easily seen from Section~2 of 
Schlickewei and Schmidt~\cite{Sch-Sch} that, if $\Phi(\pm1)\neq0$ and if the ratio of any pair of distinct roots of $\Phi(X)$ is not a root of unity,
 then the number of those $k\in\N$ for which $R_k=m$ holds is finite for any $m\in\N$. This sufficient condition has been called 
{\it nondegenerate} or {\it strictly nondegenerate} condition in literature including \cite{Sch-Sch}. 
Referring to them, we assume in this paper the following condition (ND) on $\{R_k\}_{k\geq0}$:
\begin{enumerate}
\renewcommand{\labelenumi}{(ND)}
\item $\Phi(\pm1)\neq0$, the ratio of any pair of distinct roots of $\Phi(X)$ is not a root of unity, and $\{R_k\}_{k\geq0}$ is not a geometric progression.
\end{enumerate}
The second author~\cite{Tanaka2012} proved that the infinite set $\{\varphi_3^{(l)}(\beta)\mid l\in\N,\ \beta\in\Qbar^\times\setminus\{a^{-R_k}\}_{k\geq0}\}$ is algebraically independent if $\{R_k\}_{k\geq0}$ satisfies the condition (ND). This result was obtained by using the fact, 
deduced from the logarithmic derivation, 
that $\varphi_3^{(l)}(y)$ is expressed as the polynomial with integral coefficients of $\varphi_3(y),\ \psi(y)\coloneqq\sum_{k=0}^\infty a^{R_k}/(1-a^{R_k}y),\ \psi'(y),\ldots,\psi^{(l)}(y)$ together with the fact that $\varphi_3(y)$ and $\psi^{(l)}(y)\ (l\in\N)$ are specializations of certain functions of several variables satisfying Mahler type functional equations stated in Lemmas~\ref{lem:indep_values} and \ref{lem:indep_functions}. 

The first author~\cite{Ide} solved the remaining problem on the algebraic independency of the derivatives of $\varphi_3(y)$ at its zeros $\{a^{-R_k}\}_{k\geq0}$. 
Moreover, he succeeded to treat the two variable function 
$$\varphi_4(x,y):=\sum_{k=0}^\infty a^{R_k}x^k\prod_{\substack{k'=0,\\k'\neq k}}^\infty\left(1-a^{R_{k'}}y\right),$$
for which $\varphi'_3(y)=-\varphi_3(y)\psi(y)=-\varphi_4(1,y)$ holds,
and proved that the infinite set 
$$
\left\{\frac{\partial^{l+m}\varphi_4}{\partial x^l\partial y^m}(\alpha,\beta)\relmiddle| \alpha\in\Qbar^\times,\ \beta\in\Qbar,\ l\in\N,\ m\in\Z_{>0}\right\}{\textstyle\bigcup}\left\{\varphi'_3(\beta)\relmiddle|\beta\in\Qbar\right\}
$$
is algebraically independent if $\{R_k\}_{k\geq0}$ with the condition (ND) is strictly increasing, 
where $\Z_{>0}$ denotes the set of positive integers.

Now we prepare to state our result on the case of arbitrary number of variables. 
Let $y_1,\ldots,y_s$ be complex variables and write $\y\coloneqq(y_1,\ldots,y_s)$. For an analytic function $f(\y)$ and for a vector $\m=(m_1,\ldots,m_s)\in\N^s$, we denote
$$f^{(\m)}(\y)\coloneqq\frac{\partial^{m_1+\cdots+m_s}f}{\partial y_1^{m_1}\cdots\partial y_s^{m_s}}(\y).$$
As Corollary \ref{cor:perfect} below, we obtain an entire function $\Xi(\y)$ on $\C^s$ defined by an explicit formula and having the property that the infinite set
$$
\left\{\Xi^{(\m)}(\bbeta)\relmiddle|\bbeta\in\Qbar^s,\ \m\in\N^s\right\}
$$
is algebraically independent. 
We stress here the following two points. 

One is that the points $\bbeta\in\Qbar^s$ are allowed to have zero components 
in contrast with the previous results mentioned above. 
The other is that, not only such a function $\Xi(\y)$ must be asymmetric with respect to its variables $\y$, 
but also $\Xi(\y)$ must have much stronger asymmetry, 
that is, no algebraic relation should arise under any polynomial-type transformation of the variables. 
In fact, let $P_{ij}(\y)$ $(1\leq i\leq r,\ 1\leq j\leq s)$ be any polynomials in $\Qbar[\y]$ such that the $s$-tuples $\bm{P}_i(\y)\coloneqq(P_{i1}(\y),\ldots,P_{is}(\y))$ $(1\leq i\leq r)$ are distinct.  
If the functions $\Xi(\bm{P}_i(\y))$ $(1\leq i\leq r)$ are algebraically dependent over the rational function field $\Qbar(\y)$, 
considering the differences between the components of $\bm{P}_i(\y)$ and $\bm{P}_j(\y)$ for $1\leq i<j\leq r$, 
we can choose a point $\ba\in\Qbar^s$ such that the points $\bbeta_i\coloneqq\bm{P}_i(\ba)$ $(1\leq i\leq r)$ are distinct and the values $\Xi(\bbeta_i)$ $(1\leq i\leq r)$ are algebraically dependent, which violates the required  property. 

Our strategy is summarized as follows: 
In 
Subsection~\ref{sec:2.1} 
we establish key theorem and lemmas by which 
the algebraic independence of the values of functions in question is replaced with that of 
functions easier to deal with. 
More precisely, 
we define 
the infinite products 
$g_i(y_i)=\prod_{k=0}^\infty(1-a_k^{(i)}y_i)\ (1\leq i\leq s)$ and Lambert type series $h_i(y_i)=\sum_{k=0}^\infty a_k^{(i)}/(1-a_k^{(i)}y_i)\ (1\leq i\leq s)$, which are more general than $\varphi_3(y)$ and $\psi(y)$, respectively. 
Considering the 
product 
$g(\y)=\prod_{i=1}^s g_i(y_i)$, 
multiplying it by the multiple Lambert type series $h(\y)=\sum_{k=0}^\infty\prod_{i=1}^s a_k^{(i)}/(1-a_k^{(i)}y_i)$, 
and using $g_i(y_i)\ (1\leq i\leq s)$ and $h_i(y_i)\ (1\leq i\leq s)$ as auxiliary functions,  
we can proceed the main aim of this paper. 
Namely, 
the algebraic independence of all the values of all the partial derivatives in any orders of the entire function $\theta(\y)=g(\y)h(\y)$, at all algebraic points, is inherited from that of the values of all the partial derivatives 
of $h(\y)$ in any orders at whole algebraic points except the poles of $h(\y)$. 
The latter values are algebraically independent 
if $a_k^{(i)}=a_i^{R_k}\ (1\leq i\leq s)$ with $a_1,\ldots,a_s$ different enough, 
which is deduced in Subsection~\ref{sec:2.2} from Theorem~\ref{thm:main3} below. 
The proof of Theorem~\ref{thm:main3}, stated in Subsection~\ref{sec:3.3}, 
is based on the fact that $h(\y)$ satisfies a Mahler type functional equation of several variables. 

More precisely, we use Kubota's criteria, 
stated as Lemmas~\ref{lem:indep_values} and \ref{lem:indep_functions} 
in Subsection~\ref{sec:3.1}, 
respectively on the algebraic independence of the values of Mahler functions and that of the Mahler functions themselves over the rational function field. 
By the criteria, 
if on the contrary the values of all the partial derivatives 
of $h(\y)$ in any orders at whole algebraic points except the poles of $h(\y)$ are algebraically dependent, 
then there exists a nontrivial rational function solution of a certain 
simple 
functional equation of several variables. 
Finally in Section~\ref{sec:4}, by extending the method of the second author's previous result~\cite{Tanaka1999} 
involving a generic point of the algebraic variety defined by the denominator of such a solution, 
its 
existence 
is reduced to the linear dependence of certain rational functions of $s$ variables. 
This is a contradiction 
by Lemma~\ref{lem:lin indep}, which is proved 
in Subsection~\ref{sec:3.2} 
by using a certain valuation 
on function fields of $s$ variables.

Let $a_1,\ldots,a_s$ be multiplicatively independent algebraic numbers with $0<|a_i|<1$ $(1\leq i\leq s)$. For each $i$ $(1\leq i\leq s)$, we define
\begin{equation}\label{eq:G_i H_i}
G_i(y_i)\coloneqq\prod_{k=0}^\infty \left(1-a_i^{R_k}y_i\right),\quad H_i(y_i)\coloneqq \sum_{k=0}^\infty\frac{a_i^{R_k}}{1-a_i^{R_k}y_i}
\end{equation}
and for each algebraic number $\beta$, we define
\begin{equation}\label{eq:N_i}
N_{i,\beta}\coloneqq\#\{k\geq 0 \mid a_i^{-R_k}=\beta\}=\ord_{y_i=\beta}G_i(y_i).
\end{equation}
Moreover, for each $\bbeta=(\beta_1,\ldots,\beta_s)\in\Qbar^s$, we denote
$$\cM_{\bbeta}\coloneqq\{\m=(m_1,\ldots,m_s)\in\N^s\mid m_i\geq N_{i,\beta_i}\ {\rm for}
\ 1\leq i\leq s\}.$$
Let
\begin{equation}\label{eq:G H Theta}
G(\y)\coloneqq\prod_{i=1}^s G_i(y_i),\quad H(\y)\coloneqq\sum_{k=0}^\infty\prod_{i=1}^s\frac{a_i^{R_k}}{1-a_i^{R_k}y_i},\quad \Theta(\y)\coloneqq G(\y)H(\y).
\end{equation}

\begin{mainthm}\label{thm:main}
Suppose that $\{R_k\}_{k\geq0}$ satisfies the condition {\rm(ND)}. Then the infinite set 
$$
\left\{\Theta^{(\m)}(\bbeta)\relmiddle|\bbeta\in\Qbar^s,\ \m\in\cM_{\bbeta}\right\}
$$
is algebraically independent.
\end{mainthm}

In addition to the condition {\rm(ND)}, assume that $\{R_k\}_{k\geq0}$ is strictly increasing. Then 
$N_{i,\beta}\leq 1$ for all $\beta\in\Qbar$ and so
$
\Z_{>0}^s
$ is a subset of $\cM_{\bbeta}$ for any $\bbeta\in\Qbar^s$.
Hence
$$
\left\{\Theta^{(\m)}(\bbeta)\relmiddle|\bbeta\in\Qbar^s,\ \m\in\Z_{>0}^s
\right\}
$$
is an infinite subset of
$$
\left\{\Theta^{(\m)}(\bbeta)\relmiddle|\bbeta\in\Qbar^s,\ \m\in\cM_{\bbeta}\right\}.
$$
Therefore 
Theorem \ref{thm:main} implies that the infinite set
$$
\left\{\Theta^{(\m)}(\bbeta)\relmiddle|\bbeta\in\Qbar^s,\ \m\in\Z_{>0}^s
\right\}
$$
is algebraically independent. Namely, letting
\begin{equation}\label{eq:Xi}
\Xi(\y)\coloneqq\frac{\partial^{s}\Theta}{\partial y_1\cdots\partial y_s}(\y)=G(\y)\sum_{\substack{k_1,\ldots,k_s,l\geq0,\\ k_1,\ldots,k_s\neq l}}\prod_{i=1}^s\frac{-a_i^{R_{k_i}+R_l}}{(1-a_i^{R_{k_i}}y_i)(1-a_i^{R_l}y_i)},
\end{equation}
we obtain the following

\begin{cor}\label{cor:perfect}
Suppose that $\{R_k\}_{k\geq0}$ satisfies the condition {\rm(ND)}. Assume in addition that $\{R_k\}_{k\geq0}$ is strictly increasing. Then the infinite set 
$$
\left\{\Xi^{(\m)}(\bbeta)\relmiddle|\bbeta\in\Qbar^s,\ \m\in\N^s\right\}
$$
is algebraically independent.
\end{cor}

Here we exhibit a concrete example of Corollary \ref{cor:perfect} with arbitrary number $s$ of variables. 

\begin{ex}\label{ex:Fibonacci}
{\rm
Let $p_1,\ldots,p_s$ be distinct rational primes and $\{F_k\}_{k\geq0}$ the Fibonacci numbers defined by
$$F_0=0,\quad F_1=1,\quad F_{k+2}=F_{k+1}+F_k\quad(k\geq0).$$
Putting $a_i\coloneqq p_i^{-1}$ $(1\leq i\leq s)$ and regarding $\{F_{k+2}\}_{k\geq0}$ as $\{R_k\}_{k\geq0}$, we define the function $\Xi(\y)$ by \eqref{eq:Xi}, namely,
$$
\Xi(\y)=\prod_{i=1}^s\prod_{k=2}^\infty \left(1-p_i^{-F_k}y_i\right)\sum_{\substack{k_1,\ldots,k_s,l\geq2,\\ k_1,\ldots,k_s\neq l}}\prod_{i=1}^s\frac{-p_i^{-F_{k_i}-F_l}}{(1-p_i^{-F_{k_i}}y_i)(1-p_i^{-F_l}y_i)}.
$$
Then by Corollary \ref{cor:perfect} the infinite set 
$$
\left\{\Xi^{(\m)}(\bbeta)\relmiddle|\bbeta\in\Qbar^s,\ \m\in\N^s\right\}
$$
is algebraically independent.
}
\end{ex}

In the case of $s=1$, Theorem \ref{thm:main} is deduced from the following previous result of the first author on $\Theta(y)=\varphi_4(1,y)=\varphi_3(y)\psi(y)$.

\begin{prop}[A special case of Theorem 1.7 of Ide \cite{Ide}]\label{prop:s=1}
Suppose that $\{R_k\}_{k\geq0}$ satisfies the condition {\rm(ND)}. Then, if $s=1$, then the infinite set
$$
\left\{\Theta^{(m)}(\beta)\relmiddle|\beta\in\Qbar,\ m\geq N_\beta\right\}{\textstyle\bigcup}\left\{G^{(N_\beta)}(\beta)\relmiddle|\beta\in\Qbar^\times\right\}
$$
is algebraically independent, where $N_\beta\coloneqq N_{1,\beta}$ for each $\beta\in\Qbar$.
\end{prop}

For obtaining the entire Main Theorem \ref{thm:main}, we actually show the following, which includes Theorem \ref{thm:main} with $s\geq2$.

\begin{thm}\label{thm:main2}
Suppose that $\{R_k\}_{k\geq0}$ satisfies the condition {\rm(ND)}. 
Assume in addition that $s\geq2$.
Then the infinite set 
\begin{align*}
&\left\{\Theta^{(\m)}(\bbeta)\relmiddle|\bbeta\in\Qbar^s,\ \m\in\cM_{\bbeta}\right\}\\
&{\textstyle\bigcup}\left\{G_i^{(m)}(\beta)\relmiddle|1\leq i\leq s,\ \beta\in\Qbar^\times,\ m\geq N_{i,\beta}\right\}\\
&{\textstyle\bigcup}\left\{G_i^{(m)}(0)\relmiddle|1\leq i\leq s,\ m\geq 1\right\}
\end{align*}
is algebraically independent.
\end{thm}

Theorem \ref{thm:main2} is deduced from the following theorem together with key theorem and lemmas stated in the next section.

\begin{thm}\label{thm:main3}
Suppose that $\{R_k\}_{k\geq0}$ satisfies the condition {\rm(ND)}.
Assume in addition that $s\geq2$.
Then the infinite set 
\begin{align*}
&\left\{H^{(\m)}(\bbeta)\relmiddle|\bbeta\in\cB^s,\ \m\in\N^s\right\}\\
&{\textstyle\bigcup}\left\{H_i^{(m)}(\beta)\relmiddle|1\leq i\leq s,\ \beta\in\cB,\ m\geq 0\right\}\\
&{\textstyle\bigcup}\left\{G_i(\beta)\relmiddle|1\leq i\leq s,\ \beta\in\cB\setminus\{0\}\right\}
\end{align*}
is algebraically independent, where
$$\cB\coloneqq\Qbar\setminus\bigcup_{i=1}^s\{a_i^{-R_k}\mid k\geq0\}=\{\beta\in\Qbar\mid N_{i,\beta}=0\ {\rm for}
\ 1\leq i\leq s\}.$$
\end{thm}

\section{Proof of Theorem \ref{thm:main2}}\label{sec:2}
In this section we provide the explicit invertible linear relations, 
which are represented by lower triangular matrices, 
between the values of the partial derivatives of $\theta(\y)$ and those of $h(\y)$ defined by \eqref{eq:g h theta} below. 
Then, using such relations, we deduce Theorem~\ref{thm:main2} from Theorem~\ref{thm:main3}. 

\subsection{Preparation for the proof of Theorem \ref{thm:main2}}\label{sec:2.1}
Let $\{a_k^{(i)}\}_{k\geq0}$ $(1\leq i\leq s)$ be sequences of algebraic numbers satisfying
$$
\sum_{k=0}^\infty|a_k^{(i)}|<\infty
$$
and let
\begin{equation}\label{eq:g_i h_i}
g_i(y_i)\coloneqq\prod_{k=0}^\infty \left(1-a_k^{(i)}y_i\right),\quad h_i(y_i)\coloneqq \sum_{k=0}^\infty\frac{a_k^{(i)}}{1-a_k^{(i)}y_i}
\quad\;(1\leq i\leq s).
\end{equation}
Define
\begin{equation}\label{eq:g h theta}
g(\y)\coloneqq\prod_{i=1}^s g_i(y_i),\quad h(\y)\coloneqq\sum_{k=0}^\infty\prod_{i=1}^s\frac{a_k^{(i)}}{1-a_k^{(i)}y_i},\quad \theta(\y)\coloneqq g(\y)h(\y).
\end{equation}

First we establish the 
invertible linear relations between the values 
at any point $\bbeta=(\beta_1,\ldots,\beta_s)\in\Qbar^s$ 
of the functions above and those defined by `shifted' sequences of $\{a_k^{(i)}\}_{k\geq0}$ $(1\leq i\leq s)$. 
Similarly to the numbers $N_{i,\beta_i}$ $(1\leq i\leq s)$ defined by \eqref{eq:N_i}, we define the numbers $n_i$ $(1\leq i\leq s)$ by
$$n_i\coloneqq\#\{k\geq 0 \mid a_k^{(i)}\neq0,\ (a_k^{(i)})^{-1}=\beta_i\}=\ord_{y_i=\beta_i}g_i(y_i)\quad(1\leq i\leq s).$$
Let $\bn\coloneqq(n_1,\ldots,n_s).$ Since $a_k^{(i)}\to0$ as $k$ tends to infinity for all $1\leq i\leq s$, there exists a sufficiently large integer $k_0$ such that $1-a_k^{(i)}\beta_i\neq0$ $(1\leq i\leq s)$ for all $k\geq k_0$. Put $\wa_k^{(i)}\coloneqq a_{k+k_0}^{(i)}$ $(1\leq i\leq s,\ k\geq0)$. Let $\wg_i(y_i)$, $\wh_i(y_i)$ $(1\leq i\leq s)$ and $\wg(\y)$, $\wh(\y)$, $\wt(\y)$ be the functions given respectively by \eqref{eq:g_i h_i} and \eqref{eq:g h theta} with the sequences $\{\wa_k^{(i)}\}_{k\geq0}$ $(1\leq i \leq s)$ in place of $\{a_k^{(i)}\}_{k\geq0}$ $(1\leq i \leq s)$. For each positive integer $N$, let $\cL_N(\mathcal{R})$ (resp. $\cL_N^\ast(\mathcal{R})$) be a multiplicative group of $N\times N$ lower triangular matrices with entries in the commutative ring $\mathcal{R}$ whose diagonal entries are units (resp. 1's) of $\mathcal{R}$. 
The following theorem plays a crucial role in the proof of Theorem~\ref{thm:main2}.

\begin{thm}\label{thm:linear relation}
Let $M$ be a nonnegative integer. Then, for each $i$ $(1\leq i\leq s)$, there exists $L_i\in\cL_{M+1}(\Qbar)$ such that
\begin{equation}\label{eq:linear relation 1}
\left(
\begin{array}{c}
g_i^{(n_i)}(\beta_i) \\
g_i^{(1+n_i)}(\beta_i) \\
\vdots \\
g_i^{(M+n_i)}(\beta_i)
\end{array}
\right)=L_i\left(
\begin{array}{c}
\wg_i(\beta_i) \\
\wg_i^{\,\prime}(\beta_i) \\
\vdots \\
\wg_i^{(M)}(\beta_i)
\end{array}
\right).
\end{equation}
Moreover, let $\bm{\theta}$ and $\widetilde{\bm{\theta}}$ be column vectors given by sorting the values $\theta^{(\m+\bn)}(\bbeta)$ and $\wt^{(\m)}(\bbeta)$ in an ascending lexicographical order of $\m\in{\{0,\ldots,M\}}^s$, respectively. Then there exists $L\in\cL_{{(M+1)}^s}(\Qbar)$ such that
\begin{equation}\label{eq:linear relation 2}
\bm{\theta}\equiv L\widetilde{\bm{\theta}}\pmod{W^{{(M+1)}^s}},
\end{equation}
where $W$ is the $\Qbar$-vector space generated by $\{\wg^{(\m)}(\bbeta)\mid\m\in{\{0,\ldots,M+1\}}^s\}$.
\end{thm}

\begin{rem}\label{rem:vector space W}
{\rm
For each $\m=(m_1,\ldots,m_s)\in{\{0,\ldots,M+1\}}^s$, we see by \eqref{eq:linear relation 1} that $g^{(\m+\bn)}(\bbeta)$ is represented as a linear combination of $\wg^{(\bmu)}(\bbeta)$ $(\bmu\in\{0,\ldots,m_1\}\times\cdots\times\{0,\ldots,m_s\})$ and conversely $\wg^{(\m)}(\bbeta)$ is represented as that of $g^{(\bmu+\bn)}(\bbeta)$ $(\bmu\in\{0,\ldots,m_1\}\times\cdots\times\{0,\ldots,m_s\})$. Hence $\{g^{(\m+\bn)}(\bbeta)\mid\m\in{\{0,\ldots,M+1\}}^s\}$ and $\{\wg^{(\m)}(\bbeta)\mid\m\in{\{0,\ldots,M+1\}}^s\}$ generate the same $\Qbar$-vector space $W$ in Theorem~\ref{thm:linear relation}.
}
\end{rem}

\begin{proof}[Proof of Theorem~\ref{thm:linear relation}]
Let
$$I_1\coloneqq\{i\in\{1,\ldots,s\}\mid n_i\geq1\},\quad I_2\coloneqq\{i\in\{1,\ldots,s\}\mid n_i=0\}.$$
For each $i\in\{1,\ldots,s\}$, we define the polynomials $P_i(y_i),Q_i(y_i)\in\Qbar[y_i]$ by
$$
P_i(y_i)\coloneqq
\begin{dcases}
\left(1-\beta_i^{-1}y_i\right)^{n_i}, & i\in I_1,\\
1, & i\in I_2
\end{dcases}
$$
and
$$
Q_i(y_i)\coloneqq
\begin{dcases}
\prod_{\substack{k=0\\a_k^{(i)}\neq\beta_i^{-1}}}^{k_0-1}\left(1-a_k^{(i)}y_i\right), & i\in I_1,\\
\prod_{k=0}^{k_0-1}\left(1-a_k^{(i)}y_i\right), & i\in I_2.
\end{dcases}
$$
First, for a fixed $i\in\{1,\ldots,s\}$, we construct $L_i\in\cL_{M+1}(\Qbar)$ satisfying \eqref{eq:linear relation 1}.
Since
$$g_i(y_i)=\prod_{k=0}^{k_0-1}\left(1-a_k^{(i)}y_i\right)\times\prod_{k=k_0}^\infty\left(1-a_k^{(i)}y_i\right)=P_i(y_i)Q_i(y_i)\wg_i(y_i),$$
we see that, for any $m\geq0$,
$$g_i^{(m+n_i)}(\beta_i)=\sum_{\mu=0}^{m}\binom {m+n_i}{n_i\quad m-\mu\quad\mu} p_iQ_i^{(m-\mu)}(\beta_i)\wg_i^{(\mu)}(\beta_i),$$
where $p_i\coloneqq P_i^{(n_i)}(y_i)\in\Qbar^\times$. Hence we obtain \eqref{eq:linear relation 1} by putting
$$
L_i\coloneqq\left(
\begin{array}{cccc}
p_iq_i&&&\\
&\binom{1+n_i}{1}p_iq_i&&\hsymb{0}\\
&&\ddots&\\
\hsymb{*}&&&\binom{M+n_i}{M}p_iq_i
\end{array}
\right),
$$
where $q_i\coloneqq Q_i(\beta_i)\in\Qbar^\times$.

In the rest of the proof we show \eqref{eq:linear relation 2}. Since
$$g(\y)=\left(\prod_{i=1}^s\prod_{k=0}^{k_0-1}\left(1-a_k^{(i)}y_i\right)\right)\wg(\y)$$
and since
$$h(\y)=\wh(\y)+\sum_{k=0}^{k_0-1}\prod_{i=1}^s\frac{a_k^{(i)}}{1-a_k^{(i)}y_i},$$
we have a decomposition
$$
\theta(\y)=\theta_1(\y)+\theta_2(\y),
$$
where
$$\theta_1(\y)\coloneqq\left(\prod_{i=1}^s\prod_{k=0}^{k_0-1}\left(1-a_k^{(i)}y_i\right)\right)\wt(\y)$$
and
$$\theta_2(\y)\coloneqq\left(\sum_{k=0}^{k_0-1}\prod_{i=1}^sa_k^{(i)}\prod_{\substack{k'=0\\k'\neq k}}^{k_0-1}\left(1-a_{k'}^{(i)}y_i\right)\right)\wg(\y).$$
Then we have
$$\bm{\theta}=\bm{\theta}_1+\bm{\theta}_2,$$
where $\bm{\theta}_1$ and $\bm{\theta}_2$ are column vectors given respectively by sorting the values $\theta_1^{(\m+\bn)}(\bbeta)$ and $\theta_2^{(\m+\bn)}(\bbeta)$ in the same ascending lexicographical order of $\m\in{\{0,\ldots,M\}}^s$ applied to $\bm{\theta}$ and $\widetilde{\bm{\theta}}$. In order to prove \eqref{eq:linear relation 2}, it is enough to show that
\begin{equation}\label{eq:theta_1}
\bm{\theta}_1=L\widetilde{\bm{\theta}},\quad L\in\cL_{{(M+1)}^s}(\Qbar)
\end{equation}
and
\begin{equation}\label{eq:theta_2}
\bm{\theta}_2\in W^{{(M+1)}^s}.
\end{equation}
Since
$$\theta_1(\y)=\left(\prod_{i=1}^sP_i(y_i)Q_i(y_i)\right)\wt(\y),$$
we have
\begin{equation}\label{eq:lower}
\theta_1^{(\m+\bn)}(\bbeta)=\sum_{\substack{\bmu=(\mu_1,\ldots,\mu_s),\\0\leq \mu_i\leq m_i\ (1\leq i\leq s)}}\left(\prod_{i=1}^s\binom{m_i+n_i}{n_i\quad m_i-\mu_i\quad\mu_i}p_iQ_i^{(m_i-\mu_i)}(\beta_i)\right)\wt^{(\bmu)}(\bbeta)
\end{equation}
for any $\m=(m_1,\ldots,m_s)\in\N^s$
. Here we note that, if $\m=(m_1,\ldots,m_s)\in{\{0,\ldots,M\}}^s$ and $\bmu=(\mu_1,\ldots,\mu_s)\in{\{0,\ldots,M\}}^s$ satisfy $\mu_i\leq m_i$ $(1\leq i\leq s)$, then $\m$ is equal to or larger than $\bmu$ with respect to any ascending lexicographical order of ${\{0,\ldots,M\}}^s$. Hence the term $(\prod_{i=1}^s\binom{m_i+n_i}{m_i}p_iq_i)\wt^{(\m)}(\bbeta)$ is of the largest index $\m$ 
among those $\bmu$ appearing in the right-hand side of \eqref{eq:lower}, which implies \eqref{eq:theta_1}. 

Next, we define
$$U_i(y_i)\coloneqq\left(1-\beta_i^{-1}y_i\right)^{n_i-1}\in\Qbar[y_i]$$
for each $i\in I_1$ and let
$$V(\y)\coloneqq\left(\sum_{k=0}^{k_0-1}\prod_{i=1}^sa_k^{(i)}\prod_{\substack{k'=0\\k'\neq k}}^{k_0-1}\left(1-a_{k'}^{(i)}y_i\right)\right)\prod_{i\in I_1}U_i(y_i)^{-1}\in\Qbar[\y].$$
Then we have
$$\theta_2(\y)=\left(\prod_{i\in I_1}U_i(y_i)\right)V(\y)\wg(\y).$$
Hence, for any $\m=(m_1,\ldots,m_s)\in\N^s$, we have
\begin{align*}
\theta_2^{(\m+\bn)}(\bbeta)&=\sum_{\substack{\bmu=(\mu_1,\ldots,\mu_s),\\0\leq\mu_i\leq m_i'\ (1\leq i\leq s)}}\left(\prod_{i\in I_1}\binom{m_i+n_i}{n_i-1\quad m_i+1-\mu_i\quad\mu_i}u_i\right)\left(\prod_{i\in I_2}\binom{m_i}{\mu_i}\right)\\
&\qquad\qquad\qquad\qquad\times V^{(\m'-\bmu)}(\bbeta)\wg^{(\bmu)}(\bbeta),
\end{align*}
where 
$$
m_i'\coloneqq
\begin{dcases}
m_i+1, & i\in I_1,\\
m_i, & i\in I_2,
\end{dcases}
$$
$\m'\coloneqq(m_1',\ldots,m_s')$, and $u_i\coloneqq U_i^{(n_i-1)}(y_i)\in\Qbar^\times$ $(i\in I_1)$. This implies \eqref{eq:theta_2} and the theorem is proved.
\end{proof}

Next we show Lemmas \ref{lem:alg relation} and \ref{lem:alg relation2} below, which assert that the functions defined respectively by (\ref{eq:g h theta}) and (\ref{eq:g_i h_i}) satisfy `invertible' algebraic relations.

\begin{lem}\label{lem:alg relation}
Let $M$ be a nonnegative integer. Let $\widehat{\bm{\theta}}(\y)$ and $\widehat{\bm{h}}(\y)$ be column vectors given by sorting the functions $\theta^{(\m)}(\y)/g(\y)$ and $h^{(\m)}(\y)$ in an ascending lexicographical order of $\m\in{\{0,\ldots,M\}}^s$, respectively. Then there exists $A\in\cL^\ast_{{(M+1)}^s}(\Z[\{g_i^{(m)}(y_i)/g_i(y_i)\mid 1\leq i\leq s,\ 1\leq m\leq M\}])$ such that
$$\widehat{\bm{\theta}}(\y)=A\,\widehat{\bm{h}}(\y).$$
\end{lem}

\begin{proof}
From \eqref{eq:g h theta}, we obtain
$$
\frac{\theta^{(\m)}(\y)}{g(\y)}=\sum_{\substack{\bmu=(\mu_1,\ldots,\mu_s),\\0\leq\mu_i\leq m_i\ (1\leq i\leq s)}}\left(\prod_{i=1}^s\binom{m_i}{\mu_i}\frac{g_i^{(m_i-\mu_i)}(y_i)}{g_i(y_i)}\right)h^{(\bmu)}(\y)
$$
for any $\m=(m_1,\ldots,m_s)\in{\{0,\ldots,M\}}^s$. Hence, in a similar way to the argument stated immediately after \eqref{eq:lower} in the proof of Theorem~\ref{thm:linear relation}, we obtain the equation in the lemma. 
\end{proof}

\begin{lem}\label{lem:alg relation2}
Let $M$ be a positive integer. Then, for each $i$ $(1\leq i\leq s)$, there exist $B_i\in\cL^\ast_{M+1}(\Z[\{h_i^{(m)}(y_i)\mid 0\leq m\leq M-1\}])$ and $C_i\in\cL^\ast_{M+1}(\Z[\{g_i^{(m)}(y_i)/g_i(y_i)\mid 1\leq m\leq M\}])$ such that
$$
\left(
\begin{array}{c}
-g_i'(y_i)/g_i(y_i) \\
-g_i''(y_i)/g_i(y_i) \\
\vdots \\
-g_i^{(M+1)}(y_i)/g_i(y_i)
\end{array}
\right)=B_i\left(
\begin{array}{c}
h_i(y_i) \\
h_i'(y_i) \\
\vdots \\
h_i^{(M)}(y_i)
\end{array}
\right)
$$
and
$$
\left(
\begin{array}{c}
h_i(y_i) \\
h_i'(y_i) \\
\vdots \\
h_i^{(M)}(y_i)
\end{array}
\right)=C_i\left(
\begin{array}{c}
-g_i'(y_i)/g_i(y_i) \\
-g_i''(y_i)/g_i(y_i) \\
\vdots \\
-g_i^{(M+1)}(y_i)/g_i(y_i)
\end{array}
\right).
$$
\end{lem}

\begin{proof}
Since $g_i'(y_i)=-g_i(y_i)h_i(y_i)$, we see inductively that, for any $m\geq0$,
\begin{equation}\label{eq:P_m}
g_i^{(m)}(y_i)=g_i(y_i)P_m\!\left(h_i(y_i),\ldots,h_i^{(m-1)}(y_i)\right)\qquad(1\leq i\leq s),
\end{equation}
where $P_0\coloneqq1$ and $P_m(X_0,\ldots,X_{m-1})\in\Z[X_0,\ldots,X_{m-1}]$ $(m\geq1)$. Then again from the equation $-g_i'(y_i)=g_i(y_i)h_i(y_i)$, using the Leibniz rule and \eqref{eq:P_m}, we have
\begin{align}\label{eq:g_i and h_i}
-\frac{g_i^{(m+1)}(y_i)}{g_i(y_i)}&=\frac{1}{g_i(y_i)}\sum_{\mu=0}^m\binom{m}{\mu}g_i^{(m-\mu)}(y_i)h_i^{(\mu)}(y_i)\nonumber\\
&=\sum_{\mu=0}^m\binom{m}{\mu}P_{m-\mu}\!\left(h_i(y_i),\ldots,h_i^{(m-\mu-1)}(y_i)\right)h_i^{(\mu)}(y_i)\qquad(1\leq i\leq s)
\end{align}
for any $m\geq0$. On the other hand, we see inductively that, for any $m\geq0$,
$$
\frac{d^m}{dy_i^m}\left(\frac{1}{g_i(y_i)}\right)=\frac{1}{g_i(y_i)}Q_m\!\left(-\frac{g_i'(y_i)}{g_i(y_i)},\ldots,-\frac{g_i^{(m)}(y_i)}{g_i(y_i)}\right)\qquad(1\leq i\leq s),
$$
where $Q_0\coloneqq1$ and $Q_m(Y_1,\ldots,Y_m)\in\Z[Y_1,\ldots,Y_m]$ $(m\geq1)$. Since $h_i(y_i)=-g_i'(y_i)/g_i(y_i)$, we have
\begin{align}\label{eq:h_i and g_i}
h_i^{(m)}(y_i)&=-\sum_{\mu=0}^m\binom{m}{\mu}g_i^{(\mu+1)}(y_i)\frac{d^{m-\mu}}{dy_i^{m-\mu}}\left(\frac{1}{g_i(y_i)}\right)\nonumber\\
&=\sum_{\mu=0}^m\binom{m}{\mu}Q_{m-\mu}\!\left(-\frac{g_i'(y_i)}{g_i(y_i)},\ldots,-\frac{g_i^{(m-\mu)}(y_i)}{g_i(y_i)}\right)\left(-\frac{g_i^{(\mu+1)}(y_i)}{g_i(y_i)}\right)\nonumber\\
&\qquad\qquad\qquad\qquad\qquad\qquad\qquad\qquad\qquad\qquad\qquad(1\leq i\leq s)
\end{align}
for any $m\geq0$. Hence, letting
$$B(X_0,\ldots,X_{M-1})\coloneqq\left(
\begin{array}{cccc}
1&&&\\
P_1(X_0)&1&&\hsymb{0}\\
\vdots&\ddots&\quad\ddots&\\
P_M(X_0,\ldots,X_{M-1})&\cdots&\binom{M}{M-1}P_1(X_0)&1
\end{array}
\right),$$
$$B_i\coloneqq B(h_i(y_i),\ldots,h_i^{(M-1)}(y_i))\qquad(1\leq i\leq s),$$
$$C(Y_1,\ldots,Y_M)\coloneqq\left(
\begin{array}{cccc}
1&&&\\
Q_1(Y_1)&1&&\hsymb{0}\\
\vdots&\ddots&\quad\ddots&\\
Q_M(Y_1,\ldots,Y_M)&\cdots&\binom{M}{M-1}Q_1(Y_1)&1
\end{array}
\right),$$
and
$$C_i\coloneqq C\left(-\frac{g_i'(y_i)}{g_i(y_i)},\ldots,-\frac{g_i^{(M)}(y_i)}{g_i(y_i)}\right)\qquad(1\leq i\leq s),$$
we obtain the equations in the lemma from \eqref{eq:g_i and h_i} and \eqref{eq:h_i and g_i}.
\end{proof}

\subsection{Proof of Theorem \ref{thm:main2}}\label{sec:2.2}

Using Theorem~\ref{thm:linear relation} together with Lemmas \ref{lem:alg relation} and \ref{lem:alg relation2}, we deduce Theorem~\ref{thm:main2} from Theorem~\ref{thm:main3}. 

\begin{proof}[Proof of Theorem \ref{thm:main2}]
Let $\beta_0\coloneqq0$ and let $\beta_1,\ldots,\beta_J$ be any nonzero distinct algebraic numbers. For the simplicity we denote $N_{i,j}\coloneqq N_{i,\beta_j}$ $(1\leq i\leq s,\ 0\leq j\leq J)$. For each $\bj=(j_1,\ldots,j_s)\in{\{0,\ldots,J\}}^s$, let $\bbeta_{\bj}\coloneqq(\beta_{j_1},\ldots,\beta_{j_s})$ and $\bN_{\bj}\coloneqq(N_{1,j_1},\ldots,N_{s,j_s})$. In order to prove Theorem \ref{thm:main2}, it is enough to prove that, for any sufficiently large nonnegative integer $M$, the finite set
\begin{align*}
S\coloneqq&\left\{G_i^{(m)}(0)\relmiddle|1\leq i\leq s,\ 1\leq m\leq M+1\right\}\\
&{\textstyle\bigcup}\left\{G_i^{(m+N_{i,j})}(\beta_j)\relmiddle|1\leq i\leq s,\ 1\leq j\leq J,\ 0\leq m\leq M+1\right\}\\
&{\textstyle\bigcup}\left\{\Theta^{(\m+\bN_{\bj})}(\bbeta_{\bj})\relmiddle|\bj\in{\{0,\ldots,J\}}^s,\ \m\in{\{0,\ldots,M\}}^s\right\}
\end{align*}
is algebraically independent. 
(Here we take into account all the symmetric points with respect to every axis to treat $i$ and $j$ as independent indices.)
Since $R_k\to\infty$ as $k$ tends to infinity, there exists a sufficiently large integer $k_0$ such that $1-a_i^{R_k}\beta_j\neq0$ $(1\leq i\leq s,\ 1\leq j\leq J)$ for all $k\geq k_0$. Let $\wR_k\coloneqq R_{k+k_0}$ $(k\geq0)$. 
Clearly, the linear recurrence $\{\wR_k\}_{k\geq0}$ also satisfies the condition (ND) stated in Section \ref{sec:1}. Let $\wG_i(y_i)$, $\wH_i(y_i)$ $(1\leq i\leq s)$ and $\wG(\y)$, 
$\wH(\y)$, $\wT(\y)$ be the functions given respectively by \eqref{eq:G_i H_i} and \eqref{eq:G H Theta} with $\{\wR_k\}_{k\geq0}$ in place of $\{R_k\}_{k\geq0}$. 
Let 
\begin{align*}
\Lambda\coloneqq&\left\{\wG_i(\beta_j)\relmiddle|1\leq i\leq s,\ 1\leq j\leq J\right\},\\
T\coloneqq&\left\{\wG_i^{(m)}(0)\relmiddle|1\leq i\leq s,\ 1\leq m\leq M+1\right\}\\
&\;{\textstyle\bigcup}\left\{\wG_i^{(m)}(\beta_j)\relmiddle|1\leq i\leq s,\ 1\leq j\leq J,\ 0\leq m\leq M+1\right\}\\
&\;{\textstyle\bigcup}\left\{\wT^{(\m)}(\bbeta_{\bj})\relmiddle|\bj\in{\{0,\ldots,J\}}^s,\ \m\in{\{0,\ldots,M\}}^s\right\}\\
=&\;\Lambda\,{\textstyle\bigcup}\left\{\wG_i^{(m)}(\beta_j)\relmiddle|1\leq i\leq s,\ 0\leq j\leq J,\ 1\leq m\leq M+1\right\}\\
&\quad\,{\textstyle\bigcup}\left\{\wT^{(\m)}(\bbeta_{\bj})\relmiddle|\bj\in{\{0,\ldots,J\}}^s,\ \m\in{\{0,\ldots,M\}}^s\right\},\\
U\coloneqq&\;\Lambda\,{\textstyle\bigcup}\left\{\wG_i^{(m)}(\beta_j)\relmiddle|1\leq i\leq s,\ 0\leq j\leq J,\ 1\leq m\leq M+1\right\}\\
&\quad\,{\textstyle\bigcup}\left\{\wH^{(\m)}(\bbeta_{\bj})\relmiddle|\bj\in{\{0,\ldots,J\}}^s,\ \m\in{\{0,\ldots,M\}}^s\right\},
\end{align*}
and
\begin{align*}
V\coloneqq
\Lambda\,
&{\textstyle\bigcup}\left\{\wH_i^{(m)}(\beta_j)\relmiddle|1\leq i\leq s,\ 0\leq j\leq J,\ 0\leq m\leq M\right\}\\
&{\textstyle\bigcup}\left\{\wH^{(\m)}(\bbeta_{\bj})\relmiddle|\bj\in{\{0,\ldots,J\}}^s,\ \m\in{\{0,\ldots,M\}}^s\right\}.
\end{align*}
By Theorem \ref{thm:linear relation}, there exist $L_{i,j}\in\cL_{M+2}(\Qbar)$ $(1\leq i\leq s,\ 0\leq j\leq J)$ and $L_{\bj}\in\cL_{{(M+1)}^s}(\Qbar)$ $(\bj\in{\{0,\ldots,J\}}^s)$ such that
$$
\left(
\begin{array}{c}
G_i^{(N_{i,j})}(\beta_j) \\
G_i^{(1+N_{i,j})}(\beta_j) \\
\vdots \\
G_i^{(M+1+N_{i,j})}(\beta_j)
\end{array}
\right)=L_{i,j}\left(
\begin{array}{c}
\wG_i(\beta_j) \\
\wG_i'(\beta_j) \\
\vdots \\
\wG_i^{(M+1)}(\beta_j)
\end{array}
\right)
$$
and
\begin{equation}\label{eq:modulo W}
\bm{\Theta}_{\bj}\equiv L_{\bj}\widetilde{\bm{\Theta}}_{\bj}\pmod{W_{\bj}^{{(M+1)}^s}},
\end{equation}
where $\bm{\Theta}_{\bj}$ and $\widetilde{\bm{\Theta}}_{\bj}$ are column vectors given by sorting the values $\Theta^{(\m+\bN_{\bj})}(\bbeta_{\bj})$ and $\wT^{(\m)}(\bbeta_{\bj})$ in an ascending lexicographical order of $\m\in{\{0,\ldots,M\}}^s$, respectively, and $W_{\bj}$ is the $\Qbar$-vector space generated by $\{\wG^{(\m)}(\bbeta_{\bj})\mid\m\in{\{0,\ldots,M+1\}}^s\}$. Hence, noting that $G_i(0)=\wG_i(0)=1$ $(1\leq i\leq s)$, we see that $\Qbar[S]\subset\Qbar[T]$. Conversely by Remark~\ref{rem:vector space W} the $ \Qbar$-vector space $W_{\bj}$ is generated also by $\{G^{(\m+\bN_{\bj})}(\bbeta_{\bj})\mid\m\in{\{0,\ldots,M+1\}}^s\}$. Then by \eqref{eq:modulo W} we have $\widetilde{\bm{\Theta}}_{\bj}\equiv L_{\bj}^{-1}\bm{\Theta}_{\bj}\pmod{W_{\bj}^{{(M+1)}^s}}$ and hence $\Qbar[T]\subset\Qbar[S]$. Therefore $\Qbar[S]=\Qbar[T]$.

On the other hand, 
we have $\Q(T)=\Q(U)$ or more precisely,  noting that $\wG_i(\beta_j)\neq0$ $(1\leq i\leq s,\ 1\leq j\leq J)$ and  letting 
$\Lambda^{-1}\coloneqq\{1/\wG_i(\beta_j)\;|\;1\leq i\leq s,\ 1\leq j\leq J\}$, 
we have $\Q[T,\Lambda^{-1}]=\Q[U,\Lambda^{-1}]$
since 
\begin{align*}
&\Z\left[\left\{\frac{\wG_i^{(m)}(y_i)}{\wG_i(y_i)}\relmiddle|1\leq i\leq s,\ 1\leq m\leq M\right\}{\textstyle\bigcup}\left\{\frac{\wT^{(\m)}(\y)}{\wG(\y)}\relmiddle|\m\in{\{0,\ldots,M\}}^s\right\}\right]\\
=\ &\Z\left[\left\{\frac{\wG_i^{(m)}(y_i)}{\wG_i(y_i)}\relmiddle|1\leq i\leq s,\ 1\leq m\leq M\right\}{\textstyle\bigcup}\left\{\wH^{(\m)}(\y)\relmiddle|\m\in{\{0,\ldots,M\}}^s\right\}\right]
\end{align*}
by Lemma~\ref{lem:alg relation}. 
Moreover, using Lemma~\ref{lem:alg relation2}, we see that $\Q(U)=\Q(V)$ or more precisely $\Q[U,\Lambda^{-1}]=\Q[V,\Lambda^{-1}]$. 
From the facts that $\Qbar[S]=\Qbar[T]$, $\Q(T)=\Q(U)=\Q(V)$, and $\#S=\#T=\#U=\#V$, the algebraic independency of $S$ is equivalent to that of $V$. This concludes the proof since Theorem~\ref{thm:main3} for the linear recurrence $\{\wR_k\}_{k\geq0}$ asserts that $V$ is algebraically independent.
\end{proof}

\section{Proof of Theorem \ref{thm:main3}}\label{sec:3}
In this section we explain the lemmas necessary for proving Theorem \ref{thm:main3} before its proof.

\subsection{Preparation for the proof of Theorem \ref{thm:main3}}
\label{sec:3.1}
We denote by $F(z_1,\ldots,z_n)$ and by $F[[z_1,\ldots,z_n]]$ the field of rational functions and the ring of formal power series in variables $z_1,\ldots,z_n$ with coefficients in a field $F$, respectively.
The following lemma is used both in this section and in the next section. 

\begin{lem}[Nishioka \cite{Nishioka1996}]\label{lem:Puiseux}
Let $C$ be a field and $F$ a subfield of $C$. If
$$f(\z)\in C[[z_1,\ldots,z_n]]\cap F(z_1,\ldots,z_n),$$
then there exist polynomials $A(\z),B(\z)\in F[z_1,\ldots,z_n]$ such that
$$f(\z)=A(\z)/B(\z),\quad B(\0)\neq0.$$
\end{lem}

In what follows, we state the conditions required for Mahler functions. Let $\Omega=(\omega_{ij})$ be an $n\times n$ matrix with nonnegative integer entries. 
Then the maximum $\rho$ of the absolute values of the eigenvalues of $\Omega$ is itself an eigenvalue of $\Omega$ (cf. Gantmacher \cite[p. 80, Theorem 3]{Gant}). 
We define a multiplicative transformation $\Omega\colon\C^n\to\C^n$ by
\begin{equation}\label{eq:Omegaz}
\Omega\z\coloneqq \left(\prod_{j=1}^nz_j^{\omega_{1j}},\prod_{j=1}^nz_j^{\omega_{2j}},\ldots,\prod_{j=1}^nz_j^{\omega_{nj}}\right)
\end{equation}
for any $\z=(z_1,\ldots,z_n)\in\C^n$. Then $\Omega(\Omega^k\z)=\Omega^{k+1}\z$ $(k=0,1,2,\ldots)$. 

Let $\ba=(\alpha_1,\ldots,\alpha_n)$ be a point with $\alpha_1,\ldots,\alpha_n$ nonzero algebraic numbers. We assume the following four conditions on $\Omega$ and $\ba$.
\begin{enumerate}
\renewcommand{\labelenumi}{(\Roman{enumi})}
\item $\Omega$ is nonsingular and none of its eigenvalues is a root of unity, so that $\rho>1$.
\item Every entry of the matrix $\Omega^k$ is $O(\rho^k)$ as $k$ tends to infinity.
\item If we put $\Omega^k\ba\eqqcolon(\alpha_1^{(k)},\ldots ,\alpha_n^{(k)})$, then
$$
\log|\alpha_i^{(k)}|\leq-c\rho^k\quad(1\leq i\leq n) 
$$
for all sufficiently large $k$, where $c$ is a positive constant.
\item For any nonzero $f(\z)\in\C[[z_1,\ldots,z_n]]$ which converges in some neighborhood of the origin of $\C^n$, there are infinitely many positive integers $k$ such that $f(\Omega^k\ba)\neq0$.
\end{enumerate}

The following two lemmas are general criteria on the Mahler functions of several variables. 

\begin{lem}[Kubota \cite{Kubota}, see also Nishioka \cite{Nishioka}]\label{lem:indep_values}
Let $K$ be an algebraic number field. Suppose that $f_1(\z),\ldots,f_m(\z)\in K[[z_1,\ldots,z_n]]$ converge in an $n$-polydisc $U$ around the origin and satisfy the functional equations
$$f_i(\z)=a_i(\z)f_i(\Omega\z)+b_i(\z)\quad(1\leq i\leq m),$$
where $a_i(\z),b_i(\z)\in K(z_1,\ldots,z_n)$ and $a_i(\z)$ $(1\leq i\leq m)$ are defined and nonzero at the origin. Assume that the $n\times n$ matrix $\Omega$ and a point $\ba\in U$ whose components are nonzero algebraic numbers satisfy the conditions {\rm (I)--(IV)} and that $a_i(\z)$ $(1\leq i\leq m)$ are defined and nonzero at $\Omega^k\ba$ for all $k\geq0$. Then, if the values $f_1(\ba),\ldots,f_m(\ba)$ are algebraically dependent, then the functions $f_1(\z),\ldots,f_m(\z)$ are algebraically dependent over $K(z_1,\ldots,z_n)$.
\end{lem}

\begin{lem}[Kubota \cite{Kubota}, see also Nishioka \cite{Nishioka}]\label{lem:indep_functions}
Let $C$ be a field of characteristic $0$ and $\mathcal{M}$ the quotient field of $C[[z_1,\ldots,z_n]]$. Let $\Omega$ be an $n\times n$ matrix with nonnegative integer entries satisfying the condition {\rm(I)}. Suppose that $f_1(\z),\ldots,f_l(\z)\in \mathcal{M}$ satisfy the functional equations
$$f_i(\z)=f_i(\Omega\z)+b_i(\z)\quad(1\leq i\leq l),$$
where $b_i(\z)\in C(z_1,\ldots,z_n)$ $(1\leq i\leq l)$. Assume that $g_1(\z),\ldots,g_m(\z)\in \mathcal{M}^\times$ satisfy the functional equations
$$g_h(\z)=e_h(\z)g_h(\Omega\z)\quad (1\leq h\leq m),$$
where $e_h(\z)\in C(z_1,\ldots,z_n)$ $(1\leq h\leq m)$. Then, if $f_1(\z),\ldots,f_l(\z)$ and $g_1(\z),\ldots,g_m(\z)$ are algebraically dependent over $C(z_1,\ldots,z_n)$, then at least one of the following two conditions holds:
\begin{enumerate}
\item There exist $c_1,\ldots,c_l\in C$, not all zero, and $f(\z)\in C(z_1,\ldots,z_n)$ such that
\begin{equation}\label{eq: func relations1}
f(\z)=f(\Omega\z)+\sum_{i=1}^lc_ib_i(\z).
\end{equation}
\item There exist integers $d_1,\ldots,d_m$, not all zero, such that
\begin{equation}\label{eq: func relations2}
\prod_{h=1}^mg_h(\z)^{d_h}\in C(z_1,\ldots,z_n)^\times.
\end{equation}
\end{enumerate}
\end{lem}

\begin{rem}\label{rem:Adamczewski}
{\rm 
Recently, Adamczewski and Faverjon~\cite{Adamczewski}  made a big progress on  
the algebraic or linear independence of the values of Mahler functions under much more general settings 
than Lemma~\ref{lem:indep_values}. 
More precisely, they dealt with the case of the systems of functional equations $\f(\z)=A(\z)\f(\Omega\z)+\bb(\z)$, 
where $\f(\z), A(\z),$ and $\bb(\z)$ are an $m$-dimensional vector with power series components, an $m\times m$ matrix with rational function entries, and 
an $m$-dimensional vector with rational function components, respectively. 
The main result, Theorem~2.1, of \cite{Adamczewski} asserts the existence of 
a surjection from the polynomials representing algebraic relations over $\Qbar[\z]$ appearing among the $m$ components of $\f(\z)$ to those over $\Qbar$ appearing among the values of the $m$ components of $\f(\z)$  
at an algebraic point $\ba$ satisfying almost the same conditions as in Lemma~\ref{lem:indep_values}. 
On the other hand, Lemma~\ref{lem:indep_functions} provides, 
from any algebraic relation over $\Qbar(\z)$ appearing among the functions $f_1(\z),\ldots,f_l(\z)$ and $g_1(\z),\ldots,g_m(\z)$,  
the simplest relations \eqref{eq: func relations1} or \eqref{eq: func relations2} 
among 
all possible algebraic relations over $\Qbar(\z)$.  
Hence Lemma~
\ref{lem:indep_functions} 
is 
more useful than the results of \cite{Adamczewski} to reduce the algebraic dependence of the values of our functions 
to the existence of certain rational functions 
satisfying the simplest functional equations. 
}\end{rem}

The following lemma ensures a kind of uniqueness of the solutions of the Mahler type functional equations 
in the proofs of Theorems~\ref{thm:main3} and \ref{thm:main4} below. 

\begin{lem}[Nishioka \cite{Nishioka1996}]\label{lem:const}
Let $C$ be a field and $\Omega$ an $n\times n$ matrix with nonnegative integer entries satisfying the condition {\rm(I)}. Then, if an element $f$ of the quotient field of $C[[z_1,\ldots,z_n]]$ satisfies the constant coefficient equation 
$$f(\Omega\z)=af(\z)+b \qquad (a,b\in C),$$
then $f\in C$.
\end{lem}

In the rest of this section, we consider the following restricted $\Omega$. Let
\begin{equation}\label{eq:Omega}
\Omega_1\coloneqq\left(
\begin{array}{ccccc}
c_1&1&0&\cdots&0\\
c_2&0&1&\ddots&\vdots\\
\vdots&\vdots&\ddots&\ddots&0\\
\vdots&\vdots&&\ddots&1\\
c_n&0&\cdots&\cdots&0
\end{array}
\right),
\end{equation}
where $c_1,\ldots,c_n$ are the coefficients of the polynomial $\Phi(X)$ defined by \eqref{eq:Phi(X)}. 
Let $s$ be a positive integer and put
\begin{equation}\label{eq:Omega'}
\Omega_2\coloneqq\diag(\underbrace{\Omega_1,\ldots,\Omega_1}_s).
\end{equation}

\begin{lem}[cf. Lemma 4 and Proof of Theorem 2 of Tanaka \cite{Tanaka1996}]\label{lem:admissible}
Suppose that $\Phi(\pm1)\neq0$ and the ratio of any pair of distinct roots of $\Phi(X)$ is not a root of unity.
Let $a_1,\ldots,a_s$ be multiplicatively independent algebraic numbers with $0<|a_i|<1$ $(1\leq i\leq s)$. Then the matrix $\Omega_2$ defined by \eqref{eq:Omega'} and the point
$$\ba\coloneqq(\underbrace{1,\ldots,1}_{n-1},a_1,\ldots,\underbrace{1,\ldots,1}_{n-1},a_s)$$
satisfy the conditions {\rm (I)--(IV)}.
\end{lem}

Let $\{R_k\}_{k\geq0}$ be a linear recurrence of nonnegative integers satisfying \eqref{eq:LRS}. Let $s$ be a positive integer and $z_{i1},\ldots,z_{in}$ $(1\leq i\leq s)$ variables, where $n$ is the length of the recurrence formula \eqref{eq:LRS}. We write $\z_i\coloneqq(z_{i1},\ldots,z_{in})$ $(1\leq i\leq s)$ and $\z\coloneqq(\z_1,\ldots,\z_s)$ in the rest of this 
section. 
We define monomials
\begin{equation}\label{eq:M}
M(\z_i)\coloneqq z_{i1}^{R_{n-1}}\cdots z_{in}^{R_0}\quad(1\leq i\leq s).
\end{equation}
Let $\Cbar$ be an algebraically closed field of characteristic $0$. The following theorem and lemma 
play crucial roles in 
the proof of Theorem \ref{thm:main3}.

\begin{thm}\label{thm:main4}
Suppose that $\{R_k\}_{k\geq0}$ satisfies the condition {\rm (ND)} stated in Section~\ref{sec:1}. Assume that $f(\z)\in\Cbar[[\z]]=\Cbar[[\z_1,\ldots,\z_s]]$ satisfies the functional equation of the form
\begin{equation}\label{eq:func eq}
f(\z)=\alpha f(\Omega_2\z)+R(M(\z_1),\ldots,M(\z_s)),
\end{equation}
where $\alpha\in\Cbar^\times$ and $R(X_1,\ldots,X_s)\in\Cbar(X_1,\ldots,X_s)$. Assume further that $R(X_1,\ldots,X_s)$ is of the form
\begin{equation}\label{eq:R}
R(X_1,\ldots,X_s)=\frac{P(X_1,\ldots,X_s)}{Q_1(X_1)\cdots Q_s(X_s)},
\end{equation}
where $P(X_1,\ldots,X_s)\in\Cbar[X_1,\ldots,X_s],$ $Q_i(X_i)\in\Cbar[X_i],$ and $Q_i(0)\neq0$ $(1\leq i\leq s)$. Then, if $f(\z)\in\Cbar(\z),$ then $f(\z)\in\Cbar$ and $R(X_1,\ldots,X_s)\in\Cbar$.
\end{thm}

\noindent 
The proof of this theorem will be postponed to the next section. 

\begin{lem}[A special case of Theorem 2 of Tanaka \cite{Tanaka1999}]\label{lem:1999}
Suppose that $\{R_k\}_{k\geq0}$ satisfies the condition {\rm (ND)} stated in Section \ref{sec:1}. Assume that $g(\z_0)=g(z_1,\ldots,z_n)$ is a nonzero element of the quotient field of $\Cbar[[z_1,\ldots,z_n]]$ satisfying the functional equation of the form
$$g(\z_0)=R(M(\z_0))g(\Omega_1\z_0),$$
where $M(\z_0)=M(z_1,\ldots,z_n)$ is the monomial defined by \eqref{eq:M}, $\Omega_1$ is the matrix defined by \eqref{eq:Omega}, and $R(X)\in\Cbar(X)$ is defined and nonzero at $X=0$. Then, if $g(\z_0)\in\Cbar(z_1,\ldots,z_n)^\times$, then $g(\z_0)\in\Cbar^\times$ and $\,R(X)=1$.
\end{lem}

\subsection{A lemma for proving Theorem \ref{thm:main3}}\label{sec:3.2}
In this subsection we prove the following lemma by using a certain valuation 
on function fields of $s$ variables. 

\begin{lem}\label{lem:lin indep}
Let $F$ be any field. Put $\beta_0\coloneqq0\in F$ and let $\beta_1,\ldots,\beta_J$ be nonzero distinct elements of $F$. Let $M$ be a nonnegative integer, $s$ a positive integer, and $X_1,\ldots,X_s$ variables. 
Then the ${(J+1)}^s{(M+1)}^s$ elements
$$\prod_{i=1}^s\left(\frac{X_i}{1-\beta_{j_i}X_i}\right)^{m_i+1}\quad\left(\begin{array}{c}(j_1,\ldots,j_s)\in{\{0,\ldots,J\}}^s,\\(m_1,\ldots,m_s)\in{\{0,\ldots,M\}}^s\end{array}\right)$$
of $F(X_1,\ldots,X_s)$ are linearly independent over $F$.
\end{lem} 

\begin{proof}
Suppose on the contrary that there exist $c_{\bj,\m}\in F$ $(\bj\in{\{0,\ldots,J\}}^s,\ \m\in{\{0,\ldots,M\}}^s)$, not all zero, such that
\begin{equation}\label{eq:lin combi}
Q(X_1,\ldots,X_s)
\coloneqq  
\sum_{\substack{\bj=(j_1,\ldots,j_s)\in{\{0,\ldots,J\}}^s,\\\m=(m_1,\ldots,m_s)\in{\{0,\ldots,M\}}^s}}c_{\bj,\m}\prod_{i=1}^s\left(\frac{X_i}{1-\beta_{j_i}X_i}\right)^{m_i+1}=0.
\end{equation}
Let $\theta_1,\ldots,\theta_s$ be positive numbers which are linearly independent over $\Q$ and let ${\boldsymbol\theta}=(\theta_1,\ldots,\theta_s)$. 
Let $\m^\ast=(m_1^\ast,\ldots,m_s^\ast)
$ 
be the only element in ${\{0,\ldots,M\}}^s$ such that 
the inner product $\m^\ast\!\cdot{\boldsymbol\theta}$ is 
the largest among 
$\m\cdot{\boldsymbol\theta}\ (\m\in{\{0,\ldots,M\}}^s)$ with $c_{\bj,\m}\neq0$ for some $\bj\in{\{0,\ldots,J\}}^s$. 
Choose a $\bj^\ast=(j_1^\ast,\ldots,j_s^\ast)\in{\{0,\ldots,J\}}^s$ with $c_{\bj^\ast,\m^\ast}\neq0$ and define  
$$R_i(t)=
\left\{
\begin{array}{ll}
\beta_{j_i^\ast}^{-1}(1-t^{\theta_i}) & (j_i^\ast\geq1)\\ \\
t^{-\theta_i} & (j_i^\ast=0)
\end{array}
\right.$$ 
for $1\le i\le s$, where $t$ is a variable. 
Substituting $X_i=R_i(t)\ (1\le i\le s)$ into (\ref{eq:lin combi}), 
we have the Hahn series expansion 
\begin{align*}
Q(R_1(t),\ldots,R_s(t))
&=c_{\bj^\ast,\m^\ast}\left({\textstyle
\prod_{j_i^\ast\geq1}}\beta_{j_i^\ast}^{-m_i^\ast-1}\right)t^{-(\m^\ast+\bm{1})\cdot{\boldsymbol\theta}}\\
&\quad+\left(\mathrm{terms~with~exponents~higher~than~}-(\m^\ast+\bm{1})\cdot{\boldsymbol\theta}\right)\\
&=0
\end{align*}
with $\bm{1}=(1,\ldots,1)$, 
which is a contradiction. 
\end{proof}

\subsection{Proof of Theorem \ref{thm:main3}}\label{sec:3.3}

\begin{proof}[Proof of Theorem \ref{thm:main3}]
Assume on the contrary that there exist distinct $\beta_1,\ldots,\beta_J\in\cB\setminus\{0\}$ and a nonnegative integer $M$ such that the finite set
\begin{align*}
S\coloneqq&\left\{H^{(\m)}(\bbeta_{\bj})\relmiddle|\bj\in{\{0,\ldots,J\}}^s,\ \m\in{\{0,\ldots,M\}}^s\right\}\\
&{\textstyle\bigcup}\left\{H_i^{(m)}(\beta_j)\relmiddle|1\leq i\leq s,\ 0\leq j\leq J,\ 0\leq m\leq M\right\}\\
&{\textstyle\bigcup}\left\{G_i(\beta_j)\relmiddle|1\leq i\leq s,\ 1\leq j\leq J\right\}
\end{align*}
is algebraically dependent, where $\beta_0\coloneqq0$ and
$$\bbeta_{\bj}\coloneqq(\beta_{j_1},\ldots,\beta_{j_s})\quad(\bj=(j_1,\ldots,j_s)\in{\{0,\ldots,J\}}^s).$$ 
Define
\begin{align*}
&h_{\bj,\m}(\z)=h_{\bj,\m}(\z_1,\ldots,\z_s)\\
&\coloneqq\sum_{k=0}^\infty\prod_{i=1}^s\left(\frac{M(\Omega_1^k\z_i)}{1-\beta_{j_i}M(\Omega_1^k\z_i)}\right)^{m_i+1}\quad\left(\begin{array}{c}\bj=(j_1,\ldots,j_s)\in{\{0,\ldots,J\}}^s,\\\m=(m_1,\ldots,m_s)\in{\{0,\ldots,M\}}^s\end{array}\right),
\end{align*}
$$
h_{i,j,m}(\z)=h_{i,j,m}(\z_i)\coloneqq\sum_{k=0}^\infty\left(\frac{M(\Omega_1^k\z_i)}{1-\beta_{j}M(\Omega_1^k\z_i)}\right)^{m+1}\quad\left(\begin{array}{c}1\leq i\leq s,\ 0\leq j\leq J,\\0\leq m\leq M\end{array}\right),
$$
and
$$g_{i,j}(\z)=g_{i,j}(\z_i)\coloneqq\prod_{k=0}^\infty \left(1-\beta_jM(\Omega_1^k\z_i)\right)\quad(1\leq i\leq s,\ 1\leq j\leq J),$$
where $\Omega_1$ and $M(\z_i)$ $(1\leq i\leq s)$ are defined by \eqref{eq:Omega} and \eqref{eq:M}, respectively. 
From \eqref{eq:LRS}, \eqref{eq:Omegaz}, \eqref{eq:Omega}, 
and \eqref{eq:M}, we see  by induction on $k$ that
$$M(\Omega_1^k\z_i)=z_{i1}^{R_{k+n-1}}\cdots z_{in}^{R_k}\quad(1\leq i\leq s,\ k\geq0).$$
Letting
$$\ba\coloneqq(\underbrace{1,\ldots,1}_{n-1},a_1,\ldots,\underbrace{1,\ldots,1}_{n-1},a_s),$$
we have 
$$h_{\bj,\m}(\ba)=\left(\prod_{i=1}^sm_i!\right)^{-1}H^{(\m)}(\bbeta_{\bj})\quad\left(\begin{array}{c}\bj=(j_1,\ldots,j_s)\in{\{0,\ldots,J\}}^s,\\\m=(m_1,\ldots,m_s)\in{\{0,\ldots,M\}}^s\end{array}\right),$$
$$h_{i,j,m}(\ba)=h_{i,j,m}(\underbrace{1,\ldots,1}_{n-1},a_i)=(m!)^{-1}H_i^{(m)}(\beta_j)\quad\left(\begin{array}{c}1\leq i\leq s,\ 0\leq j\leq J,\\0\leq m\leq M\end{array}\right),$$
and
$$g_{i,j}(\ba)=g_{i,j}(\underbrace{1,\ldots,1}_{n-1},a_i)=G_i(\beta_j)\quad(1\leq i\leq s,\ 1\leq j\leq J).$$
Since the set $S$ is algebraically dependent, so is the set
\begin{align*}
T\coloneqq&\{h_{\bj,\m}(\ba)\mid\bj\in{\{0,\ldots,J\}}^s,\ \m\in{\{0,\ldots,M\}}^s\}\\
&\cup\{h_{i,j,m}(\ba)\mid1\leq i\leq s,\ 0\leq j\leq J,\ 0\leq m\leq M\}\\
&\cup\{g_{i,j}(\ba)\mid1\leq i\leq s,\ 1\leq j\leq J\}.
\end{align*}
Here we have the functional equations
\begin{align}\label{eq:func_eq 1}
h_{\bj,\m}(\z)=h_{\bj,\m}(\Omega_2\z)+\prod_{i=1}^s\left(\frac{M(\z_i)}{1-\beta_{j_i}M(\z_i)}\right)^{m_i+1}\nonumber\\
\left(\begin{array}{c}\bj=(j_1,\ldots,j_s)\in{\{0,\ldots,J\}}^s,\\\m=(m_1,\ldots,m_s)\in{\{0,\ldots,M\}}^s\end{array}\right),
\end{align}
\begin{equation}\label{eq:func_eq 2}
h_{i,j,m}(\z)=h_{i,j,m}(\Omega_2\z)+\left(\frac{M(\z_i)}{1-\beta_{j}M(\z_i)}\right)^{m+1}\quad\left(\begin{array}{c}1\leq i\leq s,\ 0\leq j\leq J,\\0\leq m\leq M\end{array}\right),
\end{equation}
and
\begin{equation}\label{eq:func_eq 3}
g_{i,j}(\z)=(1-\beta_jM(\z_i))g_{i,j}(\Omega_2\z)\quad(1\leq i\leq s,\ 1\leq j\leq J),
\end{equation}
where $\Omega_2$ is defined by \eqref{eq:Omega'}. Since the set of the values $T$ is algebraically dependent, by Lemmas \ref{lem:indep_values}, \ref{lem:indep_functions}, and \ref{lem:admissible} at least one the following two cases arises:
\begin{enumerate}
\item There exist algebraic numbers $c_{\bj,\m}$ $(\bj\in{\{0,\ldots,J\}}^s,\ \m\in{\{0,\ldots,M\}}^s)$ and $c_{i,j,m}$ $(1\leq i\leq s,\ 0\leq j\leq J,\ 0\leq m\leq M)$, not all zero, and $h(\z)\in\Qbar(\z)$ such that
\begin{align}\label{eq:func_eq 4}
h(\z)=h(\Omega_2\z)+\sum_{\substack{\bj=(j_1,\ldots,j_s)\in{\{0,\ldots,J\}}^s,\\\m=(m_1,\ldots,m_s)\in{\{0,\ldots,M\}}^s}}c_{\bj,\m}\prod_{i=1}^s\left(\frac{M(\z_i)}{1-\beta_{j_i}M(\z_i)}\right)^{m_i+1}\nonumber\\
+\sum_{i=1}^s\sum_{j=0}^J\sum_{m=0}^M c_{i,j,m}\left(\frac{M(\z_i)}{1-\beta_jM(\z_i)}\right)^{m+1}.
\end{align}
\item There exist integers $d_{i,j}$ $(1\leq i\leq s,\ 1\leq j\leq J)$, not all zero, such that
$$
g(\z)\coloneqq\prod_{i=1}^s\prod_{j=1}^Jg_{i,j}(\z_i)^{d_{i,j}}\in\Qbar(\z)^\times.
$$
\end{enumerate}

Assume first that the case (i) arises. Then Lemma \ref{lem:const} together with the functional equations \eqref{eq:func_eq 1}, \eqref{eq:func_eq 2}, and \eqref{eq:func_eq 4} leads to
$$
h(\z)-\left(\sum_{\substack{\bj=(j_1,\ldots,j_s)\in{\{0,\ldots,J\}}^s,\\\m=(m_1,\ldots,m_s)\in{\{0,\ldots,M\}}^s}}c_{\bj,\m}h_{\bj,\m}(\z)+\sum_{i=1}^s\sum_{j=0}^J\sum_{m=0}^M c_{i,j,m}h_{i,j,m}(\z)\right)\in\Qbar
$$
and so $h(\z)\in\Qbar[[\z]]$. Hence by Theorem \ref{thm:main4} and the functional equation \eqref{eq:func_eq 4}, the rational function
\begin{align}\label{eq:lin combi 4}
\sum_{\substack{\bj=(j_1,\ldots,j_s)\in{\{0,\ldots,J\}}^s,\\\m=(m_1,\ldots,m_s)\in{\{0,\ldots,M\}}^s}}c_{\bj,\m}\prod_{i=1}^s\left(\frac{X_i}{1-\beta_{j_i}X_i}\right)^{m_i+1}+\sum_{i=1}^s\sum_{j=0}^J\sum_{m=0}^M c_{i,j,m}\left(\frac{X_i}{1-\beta_jX_i}\right)^{m+1}
\end{align}
in $\Qbar(X_1,\ldots,X_s)$, where $X_1,\ldots,X_s$ are variables, is an element of $\Qbar$, and so it is equal to zero. 
Fix $i\in\{1,\ldots,s\}$. Noting $s\geq2$ and substituting $X_1=\cdots=X_{i-1}=X_{i+1}=\cdots=X_s=0$ into \eqref{eq:lin combi 4}, we have
$$\sum_{j=0}^J\sum_{m=0}^M c_{i,j,m}\left(\frac{X_i}{1-\beta_jX_i}\right)^{m+1}=0.$$
Hence by Lemma \ref{lem:lin indep} with $s=1$ we see that $c_{i,j,m}=0$ for all $0\leq j\leq J$ and $0\leq m\leq M$. Since $i\in\{1,\ldots,s\}$ is arbitrary, by \eqref{eq:lin combi 4} we obtain
$$\sum_{\substack{\bj=(j_1,\ldots,j_s)\in{\{0,\ldots,J\}}^s,\\\m=(m_1,\ldots,m_s)\in{\{0,\ldots,M\}}^s}}c_{\bj,\m}\prod_{i=1}^s\left(\frac{X_i}{1-\beta_{j_i}X_i}\right)^{m_i+1}=0.$$
Then again by Lemma \ref{lem:lin indep} we have $c_{\bj,\m}=0$ for all $\bj\in{\{0,\ldots,J\}}^s$ and $\m\in{\{0,\ldots,M\}}^s$, which is a contradiction.

Assume next that the case (ii) arises. Since $g(\z)\in\Qbar[[\z]]\cap\Qbar(\z)$ and since $g(\0)=1$ with $\0$ the $sn$-dimensional zero vector, by Lemma \ref{lem:Puiseux} we can write
$$g(\z)=A(\z)/B(\z),\quad A(\z),B(\z)\in\Qbar[\z],\quad A(\0),B(\0)\neq0.$$
Picking 
$i\in\{1,\ldots,s\}$ 
for which $d_{i,j}$ $(1\leq j\leq J)$ are not all zero and letting
\begin{align*}
g^\ast(\z_i)&\coloneqq g(\0,\ldots,\0,\z_i,\0,\ldots,\0)=\prod_{j=1}^Jg_{i,j}(\z_i)^{d_{i,j}},\\
A^\ast(\z_i)&\coloneqq A(\0,\ldots,\0,\z_i,\0,\ldots,\0)\in\Qbar[\z_i]\setminus\{0\},\\
B^\ast(\z_i)&\coloneqq B(\0,\ldots,\0,\z_i,\0,\ldots,\0)\in\Qbar[\z_i]\setminus\{0\},
\end{align*}
where $\0$'s are the $n$-dimensional zero vectors, we obtain
$$g^\ast(\z_i)=A^\ast(\z_i)/B^\ast(\z_i)\in\Qbar(\z_i)^\times.$$
From \eqref{eq:func_eq 3} we see that $g^\ast(\z_i)$ satisfies the functional equation
$$g^\ast(\z_i)=\left(\prod_{j=1}^J(1-\beta_jM(\z_i))^{d_{i,j}}\right)g^\ast(\Omega_1\z_i).$$
Then by Lemma \ref{lem:1999}
$$\prod_{j=1}^J(1-\beta_jX_i)^{d_{i,j}}=1$$
holds. Taking the logarithmic derivative of this equation and then multiplying both sides by $-X_i$, we get
$$\sum_{j=1}^J\beta_jd_{i,j}\frac{X_i}{1-\beta_jX_i}=0.$$
Therefore by Lemma \ref{lem:lin indep} with $s=1$ and $M=0$, we see that $\beta_jd_{i,j}=0$ or $d_{i,j}=0$ for all $1\leq j\leq J$, which is a contradiction.
\end{proof}

\section{Proof of Theorem \ref{thm:main4}}\label{sec:4}
In what follows, $C$ and $\Cbar$ denote a field of characteristic $0$ and its algebraic closure, respectively. 
We prepare the following four lemmas to prove Theorem~\ref{thm:main4}.

\begin{lem}[Nishioka \cite{Nishioka}]\label{lem:gcd}
Let $\Omega$ be an $n\times n$ non-singular matrix with nonnegative integer entries. Then, if $A(\z),B(\z)\in C[z_1,\ldots,z_n]$, then $\gcd{(A(\Omega\z),B(\Omega\z))}=\z^{\bI}$, where $\Omega\z$ is defined by \eqref{eq:Omegaz} and $\bI\in\N^n$.
\end{lem}

\begin{lem}[Tanaka \cite{Tanaka1999}]\label{lem:coprime}
Let $\{R_k\}_{k\geq0}$ be a linear recurrence of nonnegative integers satisfying \eqref{eq:LRS}. Suppose that $\{R_k\}_{k\geq0}$ is not a geometric progression. Assume that the ratio of any pair of distinct roots of $\Phi(X)$ defined by \eqref{eq:Phi(X)} is not a root of unity. Then, if $k_1$ and $k_2$ are distinct nonnegative integers, then $M(\Omega_1^{k_1}\z)-\gamma_1$ and $M(\Omega_1^{k_2}\z)-\gamma_2$ are coprime, where $M(\z)=M(z_1,\ldots,z_n)$ is the monomial defined by \eqref{eq:M}, $\Omega_1$ is the matrix defined by \eqref{eq:Omega}, and $\gamma_1,\gamma_2\in C^\times$.
\end{lem}

\begin{lem}[Tanaka \cite{Tanaka1999}]\label{lem:not geom}
Let $\{R_k\}_{k\geq0}$ be as in Lemma $\ref{lem:coprime}$. Then the sequence $\{R_k\}_{k\geq0}$ does not satisfy the linear recurrence relation of the form
$$R_{k+l}=cR_k\quad(k\geq0),$$
where $l$ is a positive integer and $c$ is a nonzero rational number.
\end{lem}

\begin{lem}[Tanaka \cite{Tanaka1999}]\label{lem:monomial}
Let $\Omega$ be an $n\times n$ matrix with nonnegative integer entries satisfying the condition {\rm(I)} stated in Section \ref{sec:3.1}. Let $P(\z)$ be a nonzero polynomial in $C[z_1,\ldots,z_n]$. Then, if $P(\Omega\z)$ divides $P(\z)\z^{\bI}$ with $\bI\in\N^n$, then $P(\z)$ is a monomial of $z_1,\ldots,z_n$.
\end{lem}

\begin{proof}[Proof of Theorem \ref{thm:main4}]
Put $f(\z)\eqqcolon A(\z)/B(\z)$, where $A(\z)$ and $B(\z)$ are coprime polynomials in $\Cbar[\z]=\Cbar[\z_1,\ldots,\z_s]$. Then we have
\begin{align*}
&A(\z)B(\Omega_2\z)\prod_{i=1}^sQ_i(M(\z_i))\\
=\;&\alpha A(\Omega_2\z)B(\z)\prod_{i=1}^sQ_i(M(\z_i))+B(\z)B(\Omega_2\z)P(M(\z_1),\ldots,M(\z_s))
\end{align*}
by \eqref{eq:func eq} and \eqref{eq:R}. We can put $\gcd{(A(\Omega_2\z),B(\Omega_2\z))}\eqqcolon\z^{\bI}$, where $\bI\in\N^{sn}$, by Lemma \ref{lem:gcd}. Then
\begin{equation}\label{eq:divide}
B(\Omega_2\z)\mid B(\z)\z^{\bI}\prod_{i=1}^sQ_i(M(\z_i))
\end{equation}
and
\begin{equation}\label{eq:divide2}
B(\z)\mid B(\Omega_2\z)\prod_{i=1}^sQ_i(M(\z_i))
\end{equation}
in $\Cbar[\z]$.

First we prove that $f(\z)\in\Cbar[\z]$. For this purpose, we show that $B(\Omega_2\z)$ divides $B(\z)\z^{\bI}$. Otherwise, by \eqref{eq:divide}, there exists a prime factor $T(\z)\in\Cbar[\z]$ of $B(\Omega_2\z)$ such that $T(\z)$ divides $\prod_{i=1}^sQ_i(M(\z_i))$ in $\Cbar[\z]$. Then there exists $\tau$ $(1\leq \tau\leq s)$ such that $T(\z)=T(\z_\tau)\in\Cbar[\z_\tau]$ and
\begin{equation}\label{eq:divide3}
T(\z_\tau)\mid Q_\tau(M(\z_\tau))
\end{equation}
in $\Cbar[\z_\tau]$. Since $T(\z_\tau)$ divides $B(\Omega_2\z)$ and $B(\Omega_2\z)$ divides $B(\Omega_2^2\z)\prod_{i=1}^sQ_i(M(\Omega_1\z_i))$ by \eqref{eq:Omega'} and \eqref{eq:divide2},
\begin{equation}\label{eq:divide4}
T(\z_\tau)\mid B(\Omega_2^2\z)\prod_{i=1}^sQ_i(M(\Omega_1\z_i)).
\end{equation}
Since $Q_\tau(0)\neq0$, noting that $\Cbar$ is algebraically closed, we see that $T(\z_\tau)$ does not divide $Q_\tau(M(\Omega_1\z_\tau))$ by Lemma \ref{lem:coprime} with \eqref{eq:divide3}. Hence $T(\z_\tau)$ divides $B(\Omega_2^2\z)$ by \eqref{eq:divide4}. Continuing this process, we see that $T(\z_\tau)$ divides $B(\Omega_2^k\z)$ for all positive integers $k$. Let $B^\ast(\z_\tau)\coloneqq B(\0,\ldots,\0,\z_\tau,\0,\ldots,\0)\in\Cbar[\z_\tau]$, where $\0$'s are the $n$-dimensional zero vectors. Since $B^\ast(\0)\neq0$ by Lemma \ref{lem:Puiseux}, we see that  $T(\z_\tau)$ divides $B^\ast(\Omega_1^k\z_\tau)$ for all positive integers $k$. Then, letting $\bm{u}=(u_1,\ldots,u_n)$ be a generic point of the algebraic variety defined by $T(\z_\tau)$ over $\Cbar$, we have $B^\ast(\Omega_1^k\bm{u})=0$ for all positive integers $k$. Using the fact that
$${\rm trans.\,deg}_{\,\Cbar}\,\Cbar(\bm{u})=n-1,$$
we obtain a contradiction in the same way as in the proof of Theorem 1 of the second author \cite{Tanaka1999}. 
Therefore $B(\Omega_2\z)$ divides $B(\z)\z^{\bI}$. 
Then by Lemmas~\ref{lem:admissible} and \ref{lem:monomial}, $B(\z)$ is a monomial in $z_{11},\ldots,z_{sn}$. Since $A(\z)$ and $B(\z)$ are coprime in $\Cbar[\z]$ and since $f(\z)=A(\z)/B(\z)\in\Cbar[[\z]]$, we conclude that $f(\z)\in\Cbar[\z]$.

Secondly we show that $R(X_1,\ldots,X_s)\in\Cbar[X_1,\ldots,X_s]$. Since $f(\z)\in\Cbar[\z]$,
\begin{equation}\label{eq:divide5}
\prod_{i=1}^sQ_i(M(\z_i))\mid P(M(\z_1),\ldots,M(\z_s))
\end{equation}
in $\Cbar[\z]$ by \eqref{eq:func eq} and \eqref{eq:R}. If $R(X_1,\ldots,X_s)\notin\Cbar[X_1,\ldots,X_s]$, then there exist $\tau$ $(1\leq \tau\leq s)$ and $\gamma\in\Cbar^\times$ such that $X_\tau-\gamma$ divides $Q_\tau(X_\tau)$ but not $P(X_1,\ldots,X_s)$. Then we have
$$P(X_1,\ldots,X_s)=(X_\tau-\gamma)F(X_1,\ldots,X_s)+G(X_1,\ldots,X_{\tau-1},X_{\tau+1},\ldots,X_s),$$
where $F(X_1,\ldots,X_s)\in\Cbar[X_1,\ldots,X_s]$ and $G(X_1,\ldots,X_{\tau-1},X_{\tau+1},\ldots,X_s)$\newline$\in\Cbar[X_1,\ldots,X_{\tau-1},X_{\tau+1},\ldots,X_s]\setminus\{0\}$. Since $M(\z_\tau)-\gamma$ divides $P(M(\z_1),\ldots,M(\z_s))$ by \eqref{eq:divide5}, 
\begin{equation}\label{eq:divide6}
M(\z_\tau)-\gamma\mid G(M(\z_1),\ldots,M(\z_{\tau-1}),M(\z_{\tau+1}),\ldots,M(\z_s))
\end{equation}
in $\Cbar[\z]$. Noting that $G(X_1,\ldots,X_{\tau-1},X_{\tau+1},\ldots,X_s)\neq0$, we see that there exist $\bb=(b_1,\ldots,b_{\tau-1},b_{\tau+1},\ldots,b_s)\in{\Cbar}^{\,s-1}$ such that $G(\bb)\neq0$. Since $R_0,\ldots,R_{n-1}$ are not all zero, there exist $\bm{a}_1,\ldots,\bm{a}_{\tau-1},\bm{a}_{\tau+1},\ldots,\bm{a}_s\in{\Cbar}^{\,n}$ such that $M(\bm{a}_i)=b_i$ $(i\in\{1,\ldots,s\}\setminus\{\tau\})$. Therefore by \eqref{eq:divide6}
$$M(\z_\tau)-\gamma\mid G(M(\bm{a}_1),\ldots,M(\bm{a}_{\tau-1}),M(\bm{a}_{\tau+1}),\ldots,M(\bm{a}_s))=G(\bb)\in\Cbar^\times$$
in $\Cbar[\z_\tau]$, which is a contradiction.

Finally we prove that $R(X_1,\ldots,X_s)\in\Cbar$, which implies $f(\z)\in\Cbar$ by Lemma~\ref{lem:const}. On the contrary we assume that $R(X_1,\ldots,X_s)\in\Cbar[X_1,\ldots,X_s]\setminus\Cbar$. Let $\delta$ be the number of terms appearing in $f(\z)\in\Cbar[\z]$. Iterating \eqref{eq:func eq}, we get
$$f(\z)-\alpha^{2\delta+1}f(\Omega_2^{2\delta+1}\z)=\sum_{k=0}^{2\delta}\alpha^kR(M(\Omega_1^k\z_1),\ldots,M(\Omega_1^k\z_s)).$$
Then the number of terms appearing in the right-hand side is at least $2\delta+1$, since $(R_{k+n-1}:\cdots:R_k)\neq(R_{k'+n-1}:\cdots:R_{k'})$ in $P^{n-1}(\Q)$ for any distinct nonnegative integers $k$ and $k'$ by Lemma \ref{lem:not geom} and so the nonconstant terms appearing in the right-hand side never cancel one another. This is a contradiction, since the number of terms appearing in the left-hand side is at most $2\delta$, and the proof of the theorem is completed.
\end{proof}

%
%

\end{document}